\documentclass[12pt,a4paper]{amsart}
\usepackage{amsmath,amssymb, amsbsy}
\usepackage{color,psfrag}
\usepackage[dvips]{graphicx}
\usepackage{enumerate}
\usepackage{stackrel}
\usepackage[bookmarks=false]{hyperref}

\usepackage{mathtools}
\mathtoolsset{showonlyrefs}

\begin{document}
\newcommand{\dyle}{\displaystyle}
\newcommand{\R}{{\mathbb{R}}}
 \newcommand{\Hi}{{\mathbb H}}
\newcommand{\Ss}{{\mathbb S}}
\newcommand{\N}{{\mathbb N}}
\newcommand{\Rn}{{\mathbb{R}^n}}
\newcommand{\ieq}{\begin{equation}}
\newcommand{\eeq}{\end{equation}}
\newcommand{\ieqa}{\begin{eqnarray}}
\newcommand{\eeqa}{\end{eqnarray}}
\newcommand{\ieqas}{\begin{eqnarray*}}
\newcommand{\eeqas}{\end{eqnarray*}}
\newcommand{\Bo}{\put(260,0){\rule{2mm}{2mm}}\\}


\theoremstyle{plain}
\newtheorem{theorem}{Theorem} [section]
\newtheorem{corollary}[theorem]{Corollary}
\newtheorem{lemma}[theorem]{Lemma}
\newtheorem{proposition}[theorem]{Proposition}
\def\neweq#1{\begin{equation}\label{#1}}
\def\endeq{\end{equation}}
\def\eq#1{(\ref{#1})}


\theoremstyle{definition}
\newtheorem{definition}[theorem]{Definition}
\newtheorem{remark}[theorem]{Remark}

\numberwithin{figure}{section}
\newcommand{\res}{\mathop{\hbox{\vrule height 7pt width .5pt depth
0pt \vrule height .5pt width 6pt depth 0pt}}\nolimits}
\def\at#1{{\bf #1}: } \def\att#1#2{{\bf #1}, {\bf #2}: }
\def\attt#1#2#3{{\bf #1}, {\bf #2}, {\bf #3}: } \def\atttt#1#2#3#4{{\bf #1}, {\bf #2}, {\bf #3},{\bf #4}: }
\def\aug#1#2{\frac{\displaystyle #1}{\displaystyle #2}} \def\figura#1#2{ \begin{figure}[ht] \vspace{#1} \caption{#2}
\end{figure}} \def\B#1{\bibitem{#1}} \def\q{\int_{\Omega^\sharp}}
\def\z{\int_{B_{\bar{\rho}}}\underline{\nu}\nabla (w+K_{c})\cdot
\nabla h} \def\a{\int_{B_{\bar{\rho}}}}
\def\b{\cdot\aug{x}{\|x\|}}
\def\n{\underline{\nu}} \def\d{\int_{B_{r}}}
\def\e{\int_{B_{\rho_{j}}}} \def\LL{{\mathcal L}}
\def\itr{\mathrm{Int}\,}
\def\D{{\mathcal D}}
 \def\tg{\tilde{g}}
\def\A{{\mathcal A}}
\def\S{{\mathcal S}}
\def\H{{\mathcal H}}
\def\M{{\mathcal M}}
\def\T{{\mathcal T}}
\def\U{{\mathcal U}}
\def\I{{\mathcal I}}
\def\F{{\mathcal F}}
\def\J{{\mathcal J}}
\def\E{{\mathcal E}}
\def\F{{\mathcal F}}
\def\G{{\mathcal G}}
\def\HH{{\mathcal H}}
\def\W{{\mathcal W}}
\def\H{\D^{2*}_{X}}
\def\d{d^X_M }
\def\LL{{\mathcal L}}
\def\H{{\mathcal H}}
\def\HH{{\mathcal H}}
\def\itr{\mathrm{Int}\,}
\def\vah{\mbox{var}_\Hi}
\def\vahh{\mbox{var}_\Hi^1}
\def\vax{\mbox{var}_X^1}
\def\va{\mbox{var}}
\def\intp{\int_0^\pi}
\def\SS{{\mathcal S}}
\def\Y{{\mathcal Y}}
\def\length{{l_\Hi}}
\newcommand{\average}{{\mathchoice {\kern1ex\vcenter{\hrule
height.4pt width 6pt depth0pt} \kern-11pt} {\kern1ex\vcenter{\hrule height.4pt width 4.3pt depth0pt} \kern-7pt} {} {} }}
\def\weak{\rightharpoonup}
\def\det{{\rm det}}
\newcommand{\ave}{\average\int}

\title[Honest vs insider trading]{Chances for the honest\\
in honest versus insider trading}

\author[M. Elizalde, C. Escudero]{Mauricio Elizalde, Carlos Escudero}
\address{}
\email{}


\date{\today}

\begin{abstract}
We study a Black-Scholes market with a finite time horizon and two investors: an honest and an insider trader.
We analyze it with anticipating stochastic calculus in two steps. First, we recover the classical result on portfolio
optimization that shows that the expected logarithmic utility of the insider is strictly greater than that of the honest trader.
Then, we prove that, whenever the market is viable, the honest trader can get a higher logarithmic utility, and therefore
more wealth, than the insider with a strictly positive probability. Our proof relies on the analysis of
a sort of forward integral variant of the Dol\'eans-Dade exponential process. The main financial conclusion is that the
logarithmic utility is perhaps too conservative for some insiders.
\end{abstract}
\maketitle

\section{Introduction}

The goal of this work is to compare the performances of two investors, one who is ignorant about the future, the ``honest trader'', and one who
possesses privileged information about it, the ``insider''. To be more precise, let us consider a filtered probability space
$(\Omega,\mathcal{F},\mathcal{F}_s,\mathbb{P})$ in which a Brownian motion $B_s$ is defined; moreover assume $\mathcal{F}_s=\sigma\{B_u, 0 \le u \le s\}$.
We furthermore assume that the wealth of the honest investor evolves according to the stochastic differential equation
\begin{equation}
d M_s = [(1-\pi_s) r_s M_s + \pi_s \mu_s M_s]ds + \pi_s \sigma_s M_s dB_s,
\label{merton}
\end{equation}
where $M_0$ (that is, $\left. M_s \right|_{s=0}$) is assumed to be a positive real number,
$r_s$ is the risk-free rate, $\mu_s$ the expected return of the stock market, $\sigma_s$ its volatility, and
$\pi_s$ is the time-dependent investor strategy or, in other words, her portfolio.
Of course, this means we are assuming a Black-Scholes market with two assets: one is risky and the other is riskless.
For the sake of simplicity,
we take $r_s$, $\mu_s$, and $\sigma_s$ to be deterministic continuous functions on $[0,T]$, with $\sigma_s$
strictly positive for all $s$. The time interval $[0,T]$, with $T>0$, is the fixed investing period, and $0 \le s \le t < T$.

If the portfolio $\pi_s \in L^2(\Omega \times [0,t])$ is $\mathcal{F}_s-$adapted (what ultimately guarantees that our first trader is \emph{honest})
then equation~\eqref{merton} admits
the unique solution~\cite{oksendal2013}
$$
M_t= M_0 \, \exp \left\{ \int_0^t \left[(1-\pi_s) r_s + \pi_s \mu_s - \frac12 \pi_s^2 \sigma_s^2 \right]ds + \int_0^t \pi_s \sigma_s dB_s \right\}.
$$
We furthermore assume that this trader is risk-averse, and in particular we implement this via the logarithmic utility
$U(\cdot)=\log(\cdot)$.
Therefore the optimal portfolio is given by the maximization of the quantity
$$
\mathbb{E}\{U[M_t/M_0]\}=\mathbb{E}\{\log[M_t/M_0]\},
$$
or, in other words, the expected utility of the normalized wealth.
The solution to this optimization problem is classical and reads
$\pi_s= (\mu_s - r_s)/\sigma_s^2$,
as shown by Merton in his celebrated work~\cite{merton1969}. In a certain sense, this solution can be regarded as
a consequence of the zero-mean property of the It\^o integral.

Things change sharply when the investor is an insider who possesses information about the future. The mathematical description of the insider
wealth we will employ relies on the use of the forward integral, introduced by Russo and Vallois in 1993, see~\cite{russo1993forward}.
In section~\ref{forward} we provide a very concise introduction to this integral. Then, in section~\ref{secmerton}, we review the optimization problem
for the insider portfolio by means of the use of the forward integral. This optimization is done under the same conditions as in the optimization
of the honest trader portfolio, that is, using the logarithmic utility. In such a case, one finds that the expected value of the insider utility
is strictly greater than the expected utility of the honest trader, as one could perhaps have guessed \emph{a priori}.
Our main result is present in section~\ref{hvsi}. There we prove that, although the insider always gets a larger utility on average, for some realizations
the honest trader can get more utility, and therefore more wealth. The proof is based on the analysis of
a kind of forward integral version of the Dol\'eans-Dade exponential process and the construction of an explicit lower bound
for the probability of the honest trader being more successful than the insider. The behavior of this lower bound is illustrated in section~\ref{examples}.
In particular, we show that, when conveniently scaled, the lower bound presents optimal investing periods $(0,T)$; we leave as an open question
whether or not this non-monotonic behavior is characteristic of the probability itself.
Finally, in section~\ref{conclusions}, we summarize the main financial consequences of our analysis.

\section{The forward integral with respect to Brownian motion}
\label{forward}

The forward integral was introduced by Russo and Vallois in 1993~\cite{russo1993forward}; see also~\cite{russo1995generalized} and~\cite{russo2000stochastic}. It allows to integrate anticipative integrands,
and in this sense it generalizes the one by It\^o; it is,
on the other hand, genuinely different from the one introduced before by Skorokhod~\cite{skorokhod1976},
see for instance~\cite{di2009malliavin}.

\begin{definition}
A stochastic process $\phi_t,\ t\in [0,T] $, is said to be forward integrable with respect to a standard Brownian motion $W_t$,
if there exists another stochastic process $I_t$ such that
\begin{equation}
\stackbin[0\leqslant t \leqslant T]{}{\sup} \ \Bigg\vert \int_0^t \phi_s \, \dfrac{W_{s+\epsilon}-W_s}{\epsilon}\ ds - I_t\Bigg\vert \rightarrow 0\ ,\ \ \epsilon\rightarrow 0^+
\label{Forward Integrals}
\end{equation}
in probability. If such a process exists, we denote
\begin{equation}
I_t:= \int_0^t \phi_s \, d^-W_s,\ t\in[0,T],
\label{d_}
\end{equation}
the forward integral of $\phi_t$ with respect to $W_t$ over $[0,T]$.
\end{definition}

If the process $\phi_t$ is It\^o integrable,
the forward integral coincides with the Itô integral of $\phi_t$ with respect to the Brownian motion.
A forward process with respect to the standard Brownian motion, is a stochastic process of the form
\begin{equation}
J_t = x+\int_0^t u_s \, ds + \int_0^t v_s \, d^-W_s,\ \ t \in [0,T],
\label{SDEforward}
\end{equation}
where $x$ is constant,
$$\int_0^T \vert u_s \vert \, ds<\infty ,\ \ \ a.s.,$$
and $v_s$ is a forward integrable stochastic process.
We present the It\^o formula for forward integrable stochastic processes in the following theorem, proven in~\cite{russo2000stochastic}.

\begin{theorem}
\textbf{The It\^o formula for forward integrals.}
Let
$$d^-J_s = u_s \, ds\ +\ v_s \, d^-W_s$$
be a forward process, a shorthand notation of \eqref{SDEforward}. Let $\mathit{g}(s,j)\in C^{1,2}([0,T]\times \mathbb{R})$ and define
$$Z_s :=\mathit{g}(s,J_s),\ \ s\in[0,T].$$
Then $Z_s,\ s\in[0,T],$ is a forward process and
$$d^-Z_s = \dfrac{\partial\mathit{g}}{\partial s}(s,J_s) \, ds\ +\ \dfrac{\partial\mathit{g}}{\partial j}(s,J_s) \, d^-J_s\ +\ \dfrac{1}{2} \, \dfrac{\partial^2\mathit{g}}{\partial j^2}(s,J_s) \, v^2_s \, ds.$$
\label{ItoForward}
\end{theorem}

We use the forward integral in the next section to find the optimal portfolio process of an insider trader.

\section{Merton Portfolio Problem with Insider Information}
\label{secmerton}

If we let the portfolio $\pi_s$ to be built with anticipating information, then it is not adapted to the natural filtration; therefore It\^o calculus cannot be used~\cite{di2009malliavin}.
To face the problem, we use the forward integral, a more general approach than enlargement of filtrations~\cite{jeanblanc2009},
which has also been used to study insider trading~\cite{pikovsky1996anticipative},
see for instance~\cite{russo2007elements}.
For this reason it is not surprising that it has been frequently used to treat insider
information~\cite{biagini2005,di2009malliavin,draouil2015donsker,draouil2016optimal,draouil2016stochastic,ewald2011information,leon2003anticipating,
nualart2006malliavin}.
As noted in the previous section, the forward integral is genuinely different from the Skorokhod one,
and we recommend the former and not the latter that has to be used in the mathematical modeling
of insider trading in view of previous results~\cite{bastons2018triple,escudero2018,escudero2021optimal}.
Let us set the investing period as $[0,T]$, with a fixed $T>0$, and assume that the trader knows $B_T$.
This can be formalized as transforming equation~\eqref{merton} into
\begin{equation}\label{merton2}
d^- M_s = [(1-\pi_s) r_s M_s + \pi_s \mu_s M_s]ds + \pi_s \sigma_s M_s d^- B_s,
\end{equation}
where we use the notation $d^-$ for the forward integral,
according to the previous section. In this case the optimal portfolio $\pi_s \in L^2(\Omega \times [0,t])$
is assumed to be adapted to $\sigma\{B_u, 0 \le u \le s\} \vee \sigma\{B_T\}$
and forward integrable. This assures the uniqueness of the solution to equation~\eqref{merton2}~\cite{di2009malliavin},
and moreover the portfolio can be computed in closed form as in~\cite{oksendal2017} to find
$\pi_s = (\mu_s - r_s)/\sigma_s^2 + (B_T-B_s)/[\sigma_s(T-s)]$.
On one hand, the utility of the honest trader can be found using the classical It\^o formula and reads
\begin{eqnarray}\nonumber
\left. \mathbb{E}\{U[M_t/M_0]\}\right|_{\text{Honest}} &=& \int_0^t \left[(1-\pi_s) r_s + \pi_s \mu_s - \frac12 \pi_s^2 \sigma_s^2 \right]_{\pi_s=(\mu_s - r_s)/\sigma_s^2}ds
\\ \nonumber
&=& \int_0^t \left[\frac{\sigma_s^2 r_s -\mu_s r_s + r_s^2}{\sigma_s^2}
+ \frac{\mu_s^2 - r_s \mu_s}{\sigma_s^2}
- \frac12 \frac{\mu_s^2 +r_s^2 - 2 r_s \mu_s}{\sigma_s^2} \right]ds
\\ \nonumber
&=& \int_0^t \left[ r_s + \frac12 \frac{(\mu_s - r_s)^2}{\sigma_s^2} \right]ds.
\end{eqnarray}
On the other hand, the utility of the insider trader is derived by means of the It\^o formula for forward integrals (Theorem \ref{ItoForward}) and reads
\begin{eqnarray}\nonumber
\left. \mathbb{E}\{U[M_t/M_0]\} \right| _{\text{Insider}} &=& \mathbb{E} \left\{ \int_0^t \left[(1-\pi_s) r_s + \pi_s \mu_s - \frac12 \pi_s^2 \sigma_s^2
\right]_{\pi_s=(\mu_s - r_s)/\sigma_s^2+(B_T-B_s)/[\sigma_s(T-s)]}ds \right\}
\\ \nonumber
& & + \mathbb{E} \left\{ \int_0^t \left[ \pi_s \sigma_s \right]_{\pi_s=(\mu_s - r_s)/\sigma_s^2+(B_T-B_s)/[\sigma_s(T-s)]} d^-B_s \right\}
\\ \nonumber
&=& \mathbb{E} \left\{ \int_0^t \left[ r_s + \frac12 \frac{(\mu_s - r_s)^2}{\sigma_s^2} - \frac12 \frac{(B_T-B_s)^2}{(T-s)^2} \right]ds \right\}
\\ \nonumber
& & + \mathbb{E} \left\{ \int_0^t \left[ \frac{B_T-B_s}{T-s} \right]^2 ds \right\}
\\ \nonumber
&=& \mathbb{E} \left\{ \int_0^t \left[ r_s + \frac12 \frac{(\mu_s - r_s)^2}{\sigma_s^2} + \frac12 \frac{(B_T-B_s)^2}{(T-s)^2} \right]ds \right\}
\\ \nonumber
&=& \int_0^t \left[ r_s + \frac12 \frac{(\mu_s - r_s)^2}{\sigma_s^2} \right]ds + \frac12 \log \left( \frac{T}{T-t} \right),
\end{eqnarray}
as implied by
\begin{eqnarray}\nonumber
& & \mathbb{E} \left\{ \int_0^t \left[ \pi_s \sigma_s \right]_{\pi_s=(\mu_s - r_s)/\sigma_s^2+(B_T-B_s)/[\sigma_s(T-s)]} d^-B_s \right\}
\\ \nonumber
&=& \mathbb{E} \left\{ \int_0^t \left[ \frac{\mu_s - r_s}{\sigma_s} + \frac{B_T-B_s}{T-s} \right] d^-B_s \right\}
\\ \nonumber
&=& \mathbb{E} \left\{ \int_0^t \left[ \frac{\mu_s - r_s}{\sigma_s} \right] d B_s \right\}
+ \mathbb{E} \left\{ \int_0^t \left[\frac{B_T-B_s}{T-s} \right] d^-B_s \right\}
\\ \nonumber
&=& \mathbb{E} \left\{ \int_0^t \frac{1}{T-s}ds \right\} \\ \nonumber
&=& \mathbb{E} \left\{ \int_0^t \left[ \frac{B_T-B_s}{T-s} \right]^2 ds \right\},
\end{eqnarray}
where we have used the expectation of, respectively, the It\^o and the forward integrals~\cite{di2009malliavin}.
Therefore
\begin{eqnarray}\nonumber
\left. \mathbb{E}\{U[M_t/M_0]\} \right|_{\text{Insider}} &=& \left. \mathbb{E}\{U[M_t/M_0]\} \right|_{\text{Honest}}
+ \frac12 \log \left( \frac{T}{T-t} \right) \\ \nonumber
&>& \left. \mathbb{E}\{U[M_t/M_0]\} \right|_{\text{Honest}},
\end{eqnarray}
for $t >0$, a result which is financially meaningful. We also have
$$
\lim_{t \nearrow T} \left. \mathbb{E}\{U[M_t/M_0]\} \right|_{\text{Insider}} = \infty,
$$
therefore the market is not viable in this limit (in fact, this is our definition of viability: the finiteness of the expected utility).
This way we have recovered the classical results in this problem~\cite{pikovsky1996anticipative}.

\section{Honest vs insider trading sample by sample}
\label{hvsi}

The previous section shows that the insider always gets a higher expected utility than the honest trader. However, it is not clear
if the honest trader can still get a higher utility for some samples. In this section we show that this is indeed the case and,
moreover, it might happen with a positive probability for every $0 < t < T$.
Let us emphasize that this happens for the optimal portfolios derived from the maximization of the expected logarithmic utility, which are not influenced
by any notion of pathwise optimality; moreover, such a notion would require the introduction of a certain stochastic ordering
as briefly discussed in section~\ref{conclusions}.
From now on we will always consider, unless explicitly stated on the contrary, that $t \in (0,T)$.
Assuming the strategies of the previous section we find
\begin{eqnarray}\nonumber
\left. \frac{M_t}{M_0} \right|_{\text{Honest}} &=& \exp \left\{ \int_0^t \left[ r_s + \frac{(\mu_s - r_s)^2}{2\sigma_s^2} \right] ds +
\int_0^t \left[ \frac{\mu_s - r_s}{\sigma_s} \right] dB_s \right\} \\ \nonumber
\left. \frac{M_t}{M_0} \right|_{\text{Insider}} &=& \left. \frac{M_t}{M_0} \right|_{\text{Honest}} \times \\ \nonumber
&& \exp \left\{ \int_0^t \left[ \frac{B_T - B_s}{T-s} \right] d^-B_s - \frac12 \int_0^t \left[ \frac{(B_T - B_s)^2}{(T-s)^2} \right] ds \right\}.
\end{eqnarray}
This last line can be considered as a type of forward integral version of a Dol\'eans-Dade exponential process~\cite{jeanblanc2009}.
Now we state our main result, which follows from the analysis of this process.

\begin{theorem}
For every $0 < t < T$, there is a subset of the sample space with a positive measure such that,
for all samples in this subset, the honest trader gets more utility than the insider, i.~e.
$$
\mathbb{P} \left( \log \left. \dfrac{M_t}{M_0} \right|_{\text{Insider}} - \log \left. \dfrac{M_t}{M_0} \right|_{\text{Honest}} <0 \right) >0.
$$
\label{TheoremH}
\end{theorem}

\begin{proof}
Taking the difference between the logarithms of the wealth amounts of the insider and the ordinary trader, that is, the difference between the utilities,
by the previous discussion we find:
$$
\begin{array}{ll}
\log \left. \dfrac{M_t}{M_0} \right|_{\text{Insider}} - \log \left. \dfrac{M_t}{M_0} \right|_{\text{Honest}} &= \displaystyle\int_0^t \frac{B_T - B_s}{T-s} d^- B_s - \dfrac{1}{2} \displaystyle\int_0^t \frac{(B_T - B_s)^2}{(T-s)^2} ds
\\\\
&= \dfrac{B_T B_t - \frac{1}{2} (B_t^2 - t)}{T-t} - \displaystyle\int_0^t \dfrac{B_T B_s - \frac{1}{2}(B_s^2 - s)}{(T-s)^2} ds
\\\\
&\ \ \ - \displaystyle\int_0^t \dfrac{\frac{1}{2}(B_T^2 - 2B_T B_s + B_s^2)}{(T-s)^2} ds
\\\\
&= \dfrac{t + 2 B_T B_t - B_t^2}{2(T-t)} - \displaystyle\int_0^t \dfrac{B_T^2 + s}{2(T-s)^2} ds
\\\\
&= \dfrac{t + 2 B_T B_t - B_t^2}{2(T-t)} - \dfrac{1}{2} \left[ \dfrac{B_T^2}{T-s} + \dfrac{s}{T-s} + \log (T-s) \right]\Bigg\vert_0^t
\\\\
&= \dfrac{1}{2} \left[ \dfrac{-B_t^2 + 2 B_T B_t + t - B_T^2}{T-t}  + \dfrac{B_T^2}{T} - \dfrac{t}{T-t} - \log \left( \dfrac{T-t}{T} \right)\right]
\\\\
&= \dfrac{1}{2} \left[ \dfrac{B_T^2}{T} + \log \left( \dfrac{T}{T-t} \right) - \dfrac{(B_T - B_t)^2}{T-t} \right],
\end{array}
$$
where we have used the integration by parts formula for the forward integral,
which is a particular case of the It\^o lemma for this integral, see Theorem~\ref{ItoForward}.
Now, let us define
$$
H_t:= \log \left. \dfrac{M_t}{M_0} \right|_{\text{Insider}} - \log \left. \dfrac{M_t}{M_0} \right|_{\text{Honest}}.
$$
To find the probability with which the honest trader gets more wealth (and therefore more utility) than the insider we have to compute
$$
\begin{array}{ll}
\mathbb{P} (H_t <0)
&= \mathbb{P} \left(
\log \left. \dfrac{M_t}{M_0} \right|_{\text{Insider}} - \log \left. \dfrac{M_t}{M_0} \right|_{\text{Honest}} <0
 \right)
 \\\\
&= \mathbb{P} \left( \dfrac{B^2_T}{T} + \log \left( \dfrac{T}{T-t} \right) - \dfrac{(B_T - B_t)^2}{T-t} <0 \right)
\\\\
&=  \mathbb{P} \left( \dfrac{B^2_t}{T} + \dfrac{2B_t}{T}(B_T - B_t) - \dfrac{t (B_T - B_t)^2}{T (T-t)} \right.
 + \log \left( \left. \dfrac{T}{T-t} \right) <0 \right).
\end{array}
$$
We can write this equivalently as
$$
\begin{array}{ll}
\mathbb{P}(H_t <0) &= \mathbb{P} \left( \dfrac{2B_T B_t}{T-t} < - \log \left( \dfrac{T}{T-t} \right) + \dfrac{B_t^2}{T-t} + \dfrac{B_T^2}{T-t} - \dfrac{B_T^2}{T} \right)
\\\\
& \geqslant \mathbb{P} \left( B_T B_t < - \left( \dfrac{T-t}{2} \right) \log \left( \dfrac{T}{T-t} \right) \right).
\end{array}
$$
By using the fact that $B_T B_t = \left( B_T - B_t \right) B_t + B_t^2$, we can decompose this lower bound into
\begin{equation}
\begin{array}{ll}
\mathbb{P}(H_t <0) &\geqslant
\mathbb{P} \left[ \left( B_T - B_t \right) + B_t < - \left( \dfrac{T-t}{2} \right) \log \left( \dfrac{T}{T-t} \right),\ B_t > 1 \right]
\\\\
& \ \ \ + \ \mathbb{P} \left[ \left( B_T - B_t \right) + B_t > \left( \dfrac{T-t}{2} \right) \log \left( \dfrac{T}{T-t} \right),\ B_t < - 1 \right]
\\\\
&= 2 \ \mathbb{P} \left[ \left( B_T - B_t \right) + B_t < - \left( \dfrac{T-t}{2} \right) \log \left( \dfrac{T}{T-t} \right),\ B_t > 1 \right].
\end{array}
\label{Decompose}
\end{equation}
Now, for the sake of brevity, let us introduce, for fixed $t$, the random variables
\begin{equation}
\begin{array}{lll}
X_t & := B_T - B_t + L_t + 1 \sim N(L_t + 1,\ T-t),\\\\
Y_t & := B_t - 1 \sim N(-1,\ t),
\label{NormalTransf2}
\end{array}
\end{equation}
where
\begin{equation}
L_t := \dfrac{T-t}{2} \log \left( \dfrac{T}{T-t} \right),
\label{L_t}
\end{equation}
so they are clearly independent.
Then we can rewrite~\eqref{Decompose} as
\begin{equation}
\begin{array}{ll}
\mathbb{P}(H_t <0) \geqslant
2 \ \mathbb{P} \left( X_t + Y_t < 0 \ , \ Y_t > 0 \right).
\end{array}
\label{ProbRight}
\end{equation}
This last probability can be expressed as the double integral
\begin{equation}
\mathbb{P} \left( X_t + Y_t < 0 \ , \ Y_t > 0 \right)=
\int_0^{\infty} \int_{\frac{3}{4}\pi}^{\pi}\  \dfrac{r}{2 \pi \sqrt{t(T-t)}} \exp \left\lbrace - \ \dfrac{1}{2} h(r,\theta,t) \right\rbrace d\theta d r,
\label{PolarIntegral2}
\end{equation}
where
$$
h(r, \theta, t) = \dfrac{\left( r \cos \theta - L_t - 1 \right)^2}{T-t} + \dfrac{\left( r \sin \theta + 1 \right)^2}{t}.
$$
To lower bound~\eqref{PolarIntegral2} we use the triangle inequalities
\begin{equation}
\begin{array}{ll}
\ \ \left| r \cos \theta - \left( L_t + 1 \right) \right| &\leqslant \left| r \cos \theta \right| + \left| L_t + 1 \right|
\\\\
&\leqslant r + L_t + 1,
\\\\
\ \ \ \ \ \ \ \ \ \ \ \ \left| r \sin \theta + 1 \right| &\leqslant \left| r \sin \theta \right| + \left| 1 \right|
\\\\
&\leqslant \dfrac{r}{\sqrt{2}} + 1,
\end{array}
\label{triangle}
\end{equation}
to, in turn, upper bound
$$
h(r, \theta, t) \le 2 (a_t r^2 + b_t r + c_t),
$$
where
\begin{equation}
\begin{array}{ll}
a_t &:= \dfrac{t + T}{2t(T - t)},
\\\\
b_t &:= \dfrac{2t(L_t + 1) + \sqrt{2}(T - t)}{t(T - t)},
\\\\
c_t &:= \dfrac{t(L_t + 1)^2 + (T - t)}{t(T - t)};
\end{array}
\label{abc}
\end{equation}
so, in particular, this bound is uniform in $\theta$.
Now, by Young inequality
$b_t r \leqslant b_t/2 + b_t r^2/2$,
we find
\begin{equation}
a_t r^2 + b_t r + c_t \leqslant \left( a_t + \dfrac{1}{2} b_t \right) r^2 + \left( \dfrac{1}{2} b_t + c_t \right).
\end{equation}
By means of these inequalities we get
\begin{equation}
\begin{array}{ll}
\mathbb{P} \left( X_t + Y_t < 0 \ , \ Y_t > 0 \right) &\ge
\displaystyle\int_0^\infty \int_{\frac{3}{4}\pi}^\pi \dfrac{r}{2\pi\sqrt{t(T-t)}} \exp \left\lbrace - a_t r^2 - b_t r - c_t \right\rbrace d \theta dr
\\\\
&\ge \displaystyle\int_0^\infty \dfrac{r}{8 \sqrt{t(T-t)}} \exp \left\lbrace - \left( a_t + \dfrac{1}{2} b_t \right) r^2 - \left( \dfrac{1}{2} b_t + c_t \right) \right\rbrace dr
\\\\
&= \dfrac{\exp \left\lbrace - \left( \frac{1}{2} b_t + c_t \right)\right\rbrace}{16 \left( a_t + \frac{1}{2} b_t \right) \sqrt{t(T-t)}}\ .
\end{array}
\end{equation}
So, we conclude
\begin{equation}
\begin{array}{ll}
\mathbb{P}(H_t <0) \ge &
\dfrac{\sqrt{t(T-t)}}{4 \left( 3t + T + 2t L_t + \sqrt{2} (T-t)\right)}
\\\\
&\times \exp \left\lbrace - \ \dfrac{t(L_t + 1)(L_t + 2) + \left( \frac{1 + \sqrt{2}}{\sqrt{2}} \right) (T-t)}{t(T-t)} \right\rbrace,
\end{array}
\end{equation}
which is a positive quantity for all $t \in (0,T)$.
\end{proof}

\begin{remark}\label{remlob}
If we define
\begin{equation}
\begin{array}{ll}
\mathcal{L}_t := &
\dfrac{\sqrt{t(T-t)}}{4 \left( 3t + T + 2t L_t + \sqrt{2} (T-t)\right)}
\\\\
&\times \exp \left\lbrace - \ \dfrac{t(L_t + 1)(L_t + 2) + \left( \frac{1 + \sqrt{2}}{\sqrt{2}} \right) (T-t)}{t(T-t)} \right\rbrace,
\end{array}
\end{equation}
then we observe
\begin{equation}\nonumber
\lim_{t \searrow 0} \mathcal{L}_t = 0, \qquad \lim_{t \nearrow T} \mathcal{L}_t = 0.
\end{equation}
In particular, when t is approaching the horizon time, the insider is closer to the moment when she has full information;
note that in this moment the market is no longer viable.
Note also that $H_0=0$ almost surely.
\end{remark}

\section{Behavior of the lower bound}
\label{examples}

To illustrate the behavior of the lower bound, we have plotted $\LL_t$ versus $t \in (0,T)$, for $T=1,10,50$, and 100, in Figure~\ref{Fig.LowerBound4}.
We see that it grows until it finds a maximum and then it decreases monotonically.
We have conveniently scaled the plots for ease of visualization.
\begin{figure}[h]
\centering
\includegraphics[scale=1]{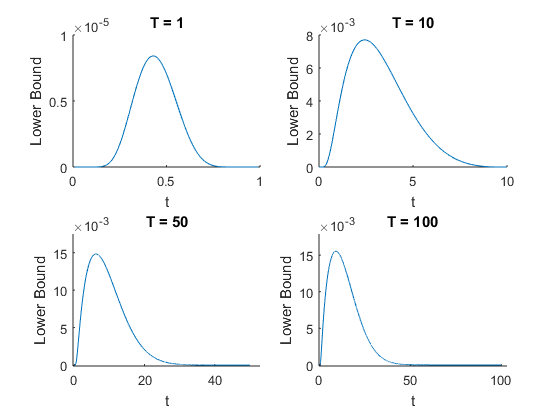}
\caption{$\LL_t$ versus $t \in (0,T)$, for different values of $T$.}
\label{Fig.LowerBound4}
\end{figure}
The illustrations in Figure~\ref{Fig.LowerBound4} show the existence of an optimal $t$, something that could be expected from Remark~\ref{remlob}.
More surprising is, perhaps, the existence of an optimal $T$ for a conveniently scaled time.
To check that, let $t = f T$ for a fixed $f \in (0,1)$; then
\begin{equation}
\begin{array}{ll}
\LL_{fT} = &
\dfrac{\sqrt{f(1-f)}}{4 \left( 3f + 1 + 2f L T + \sqrt{2} (1-f)\right)}
\\\\
&\times \exp \left\lbrace - \ \dfrac{f(L T + 1)(L T + 2) + \left( \frac{1 + \sqrt{2}}{\sqrt{2}} \right) (1-f)}{f(1-f)T} \right\rbrace,
\end{array}
\end{equation}
and
\begin{equation}\nonumber
\lim_{T \searrow 0} \LL_{fT} = 0, \qquad \lim_{T \nearrow \infty} \LL_{fT} = 0,
\end{equation}
where
$$
L := \dfrac{1-f}{2} \log \left( \dfrac{1}{1-f} \right).
$$
For the sake of visualization, we show the behavior of $\LL_{T/n}$, $n\in \N$, for different values of $T$ and $n=2,4,8,16$, in Figure~\ref{Fig.Thalf}. Observe that this is nothing but setting $f=1/n$.
\begin{figure}[h]
\centering
\includegraphics[scale=1]{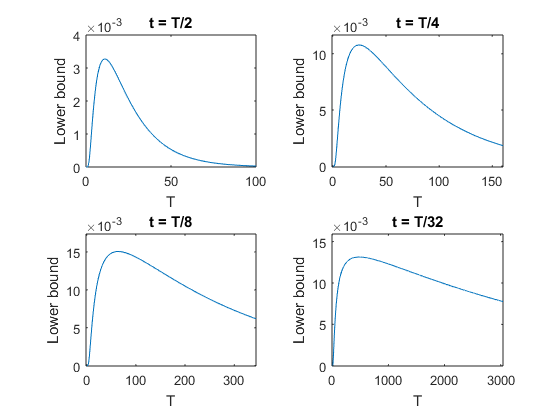}
\caption{$\LL_{T/n}$ versus $T$, for different values of $n$.}
\label{Fig.Thalf}
\end{figure}
In all the cases represented in Figure~\ref{Fig.Thalf},
the lower bound $\LL_{fT}$ has a maximum in $(0,\infty)$. To prove this is in fact a general behavior,
we take the derivative of the lower bound with respect to $T$:
$$
\dfrac{\partial \LL_{fT}}{\partial T} = \LL_{fT} \left[ - \dfrac{fL^2}{f(1-f)} + \dfrac{2f + \frac{1+\sqrt{2}}{\sqrt{2}}(1-f)}{f(1-f)T^2} - \dfrac{2fL}{(3f+1+\sqrt{2}(1-f)+2fLT)} \right].
$$
Observe that $\dfrac{\partial \LL_{fT}}{\partial T} \Big / {\LL_{fT}}$ gives us the critical points of the lower bound since $\LL_{fT}$ is positive;
moreover, all the maxima must be critical points given its smoothness and boundary behavior.
If we equate this expression to zero, we arrive at the algebraic equation
$$
\dfrac{\partial \LL_{fT}}{\partial T} \Big / {\LL_{fT}} = A_3 T^3 + A_2 T^2 + A_1 T + A_0 = 0,
$$
where
$$
\begin{array}{ll}
&A_3 = - 2 f^2 L^3 < 0,
\\\\
&A_2 = - \left( \left( 3f + 1 + \sqrt{2} (1 - f) \right) fL^2 + 2f (1-f)L \right) < 0,
\\\\
&A_1 = 2fL \left( 2f + \frac{1+\sqrt{2}}{\sqrt{2}} (1-f) \right) > 0,
\\\\
&A_0 = \left( 2f + \frac{1+\sqrt{2}}{\sqrt{2}} (1-f) \right) \left( 3f + 1 + \sqrt{2} (1 - f) \right) > 0.
\end{array}
$$
By the Descartes rule of signs, we conclude that this polynomial has exactly one root for $T>0$.
Therefore $\LL_{fT}$ has exactly one maximum. Of course, the value of $T$ that corresponds to this maximum
is explicitly computable, as any cubic equation is solvable in terms of radicals. However, the resulting
expression is cumbersome, and in consequence we do not reproduce it here.
Of course, this is just the behavior of the lower bound we have explicitly computed in the previous section.
We leave as an open question whether or not there exists such an optimal time interval for the victory of the honest
over the insider trader.

\section{Conclusions}
\label{conclusions}

In this work we have shown that, under the logarithmic utility, although the insider gets more utility on average, the honest trader still gets
more utility (and more wealth) for some samples.
The result is somehow counterintuitive, at least from the financial viewpoint, since the insider is assumed to possess strictly more information than the
honest investor (although mathematically it is clear that dominance in mean does not imply almost sure dominance). Moreover, it contrasts with previous results that studied a simplified version of this problem.
In~\cite{escudero2018,escudero2021optimal},
in which only buy-and-hold strategies with no shorting allowed were permitted, but risk aversion was replaced by risk
neutrality (for both insider and honest traders), it was shown that the insider gets more wealth (and therefore more utility) almost surely;
something more intuitive under the current assumptions.
The positive, albeit possibly small, probability of the honest trader victory
over the insider seems to be related to the logarithmic utility.
Let us try to discuss now how this may happen in simple terms.
The paths responsible of this phenomenon are those for which
$(B_T-B_t)^2/(T-t) > B_T^2/T + \log[T/(T-t)]$,
as follows from the proof of Theorem~\ref{TheoremH}. The right hand side expresses the increasing advantage of the insider when $|B_T|$ is further away
from the origin, which is the expected, and most likely, value of Brownian motion.
While the left hand side illustrates the risk aversion felt by the insider
when the intermediate value of Brownian motion has to be severely corrected in order to get the final value, which is known to the insider.
Since the honest trader is ignorant about the severe correction that the Brownian path will experience, this cannot influence her feelings.
But since the insider knows the magnitude and direction of the correction to the Brownian path,
it is doubtful that this should influence her aversion to risk much.
This behavior is presumably too conservative for an insider, at least in those cases in
which her privileged information includes all of the information managed by the ordinary trader. Perhaps in such a case the average insider
is not that risk-averse as described by a logarithmic utility.
Our present result complements a large number of works that assume this utility function. It is possible that such
an assumption underestimates the role of certain insiders in real financial markets.

It remains open to find an alternative formulation that fixes this problem. The most appealing utilities in view of our results would be,
{\it a priori}, risk-seeking power laws or even an exponential utility. However, all of these are problematic.
To see this fix $\alpha >0$ and compute the $\alpha-$th moment of the forward integral version of the Dol\'eans-Dade
exponential process from section~\ref{hvsi}:
\begin{eqnarray}\nonumber
&& \mathbb{E} \left[ \left( \left. \frac{M_t}{M_0} \right|_{\text{Insider}} \bigg/
\left. \frac{M_t}{M_0} \right|_{\text{Honest}} \right)^\alpha \, \right]
\\ \nonumber
&=& \mathbb{E} \left( \exp \left\{ \alpha \int_0^t \left[ \frac{B_T - B_s}{T-s} \right] d^-B_s
- \frac{\alpha}{2} \int_0^t \left[ \frac{(B_T - B_s)^2}{(T-s)^2} \right] ds \right\} \right)
\\ \nonumber &=&
\mathbb{E} \left( \exp \left\{ \dfrac{\alpha}{2}\left[ \dfrac{B^2_t}{T} + \dfrac{2B_t}{T}(B_T - B_t) - \dfrac{t (B_T - B_t)^2}{T (T-t)}
+ \log \left( \dfrac{T}{T-t} \right) \right] \right\} \right)
\\ \nonumber &=&
\left( \dfrac{T}{T + \alpha t} \right)^{\frac{1}{2}}
\left( \dfrac{T}{T - t} \right)^{\frac{\alpha}{2}} \left( \dfrac{T + \alpha t}{T - \alpha^2 t} \right)^{\frac{1}{2}},
\end{eqnarray}
if $t < \min\{T, T/\alpha^2\}$ and diverges otherwise,
where we have used the developments in section~\ref{hvsi} and the independence of the Brownian increments. It is clear that
this moment is well-defined for all $t \in [0,T)$ if and only if $\alpha \le 1$. This implies that risk-seeking power laws,
characterized by $\alpha >1$, and the exponential utility, make this market not viable. This of course does not mean that the
market is viable for $\alpha \le 1$, since we have not used the optimal portfolio for the corresponding power law utility, but the optimal
portfolio for the logarithmic utility, which is suboptimal in this case. Moreover, the case $\alpha=1$ is not viable in the classical
Merton problem. Although our present results suggest that risk aversion is not the best approach to model insider trading, it seems that
it is a necessary condition in order to keep the market viable, at least if the optimization of an utility is sought. Thus on one hand,
it is clear that there is room for some improvement as a risk-averse power law is still less conservative than the logarithmic utility.
On the other hand, it is perhaps more promising to try to classify the investment strategies pathwise according to some notion of stochastic order,
as done in~\cite{escudero2021optimal} in a much simplified situation, rather than to maximize the expectation of some utility; note that
a stochastic process at a fixed time is a random variable, and hence the necessity to introduce a stochastic order (see~\cite{escudero2021optimal}
for a discussion of several notions of stochastic order in a related problem).
This would have to be probably implemented together with a debt limit to the traders, in order to have a well-defined optimal portfolio
(such a limit would prevent the optimality of allocating an infinite amount of wealth in one of the investments). This is motivated by
the fact that creditors would be more risk-averse than the insider. So such an implementation requires a combination of mathematical
and financial considerations.
Finally, let us mention that the same result holds for
$\mathbb{E} \left[ \left( \left. M_t/M_0 \right|_{\text{Insider}}\right)^\alpha \, \right]$, as can be checked by means of a direct
calculation akin to the previous one after assuming the constancy of the model parameters.

\section*{Acknowledgments}

This work has been partially supported by the Government of Spain (Ministerio de Ciencia, Innovaci\'on y Universidades)
through Project PGC2018-097704-B-I00.


\bibliographystyle{abbrv}
\bibliography{Thesis_bib}


\vskip4mm
\noindent
{\footnotesize
Mauricio Elizalde\par\noindent
Departamento de Matem\'aticas\par\noindent
Universidad Aut\'onoma de Madrid\par\noindent
{\tt mauricio.elizalde@estudiante.uam.es}\par\vskip3mm\noindent
Carlos Escudero\par\noindent
Departamento de Matem\'aticas Fundamentales\par\noindent
Universidad Nacional de Educaci\'on a Distancia\par\noindent
{\tt cescudero@mat.uned.es}
}
\end{document}